\documentclass{amsart}
\usepackage{graphicx}
\usepackage{amscd}
\usepackage{amsmath}
\usepackage{amsfonts}
\usepackage{amssymb}
\theoremstyle{plain}
\newtheorem{theorem}{Theorem}
\newtheorem{corollary}[theorem]{Corollary}
\newtheorem{lemma}[theorem]{Lemma}

\theoremstyle{definition}

\newtheorem{definition}[theorem]{Definition}

\newtheorem{remark}[theorem]{Remark}
\theoremstyle{remark}

\begin{document}
\title{Cancellation elements in multiplicative lattices}

\author{Tiberiu Dumitrescu}

\address{Facultatea de Matematica si Informatica,University of Bucharest,14 A\-ca\-de\-mi\-ei Str., Bucharest, RO 010014,Romania}
\email{tiberiu\_dumitrescu2003@yahoo.com}

\thanks{2020 Mathematics Subject Classification: Primary 06B23, Secondary 13F05.}
\thanks{Key words and phrases: multiplicative lattice, cancellation element, cancellation ideal, Abstract Ideal Theory.}

\begin{abstract}\noindent
We extend to multiplicative lattices a theorem of Anderson and Roitman characterizing the cancellation ideals of a commutative ring.
\end{abstract}

\maketitle

Anderson and Roitman \cite[Theorem]{AR} showed that an ideal $Q$ of a commutative ring $R$  is a {\em cancellation ideal} (i.e. $QJ=QH$ with $J,H$ ideals implies $J=H$) iff all localizations of $Q$ at the maximal ideals of $R$  are principal   regular ideals. Subsequently Chang \cite[Theorem 2.3]{C} extended this result in the context of star operations on integral domains.
The aim of this short note is to give an Abstract Ideal Theory version of Anderson-Roitman theorem which covers also Chang's result (see 
 Theorem \ref{7} and {Remark} \ref{11}).

We recall some basic facts of Abstract Ideal Theory (for details see \cite{A}, \cite{AJ}, \cite{AJ1}, \cite{D} and    \cite{JJ}).
 A {\em (multiplicative) lattice} is a complete lattice  $(L,\leq)$ (with bottom element $0$ and top element $1$) which is also  a commutative monoid with identity $1$ (the top element) such that
 $$a( \bigvee_\alpha b_\alpha) = \bigvee_\alpha (ab_\alpha)  \mbox{ \ for all } a,b_\alpha\in L.$$
Lattice $L$ is {\em modular} if $a,b,c\in L$ and $a\geq b$ implies 
$a\wedge (b\vee c)= b \vee (a\wedge c)$.
An element $c\in L$ is {\em compact} if 
 \begin{center}
  $c\leq \bigvee S$ with $S\subseteq L$  implies $c\leq \bigvee T$ for some finite subset $T\subseteq S.$
\end{center}
\noindent
Suppose that $1$ is compact. Then every  element $a\neq 1$ is smaller than   some {\em maximal} element $m$ (i.e. maximal in $L-\{1\}$). 
An element $p\in L-\{1\}$ is {\em prime} if  
$$ab \leq p \mbox{ with } a,b \in L \mbox{ implies } a \leq p \mbox{ or } b\leq p.$$   
Every maximal element is prime. The {\em spectrum} of $L$ is the set of all prime elements of $L$.
For  $a,b\in L$,   $(a : b)$ is defined as $\bigvee \{c \in L\ | \ bc \leq a\}$.
An element $x\in L$ is {\em principal} if it is both 
{\em meet principal}, i.e. satisfies
$$a \wedge bx = ((a : x) \wedge b) x  \ \ \  \mbox{ for all }  a, b \in L $$
and 
{\em join principal}, i.e. satisfies
$$a \vee (b : x) = ((ax \vee b) : x) \ \ \mbox{ for all }  a, b \in L .$$
If $x$ and $y$ are principal elements, then so is $xy$. 
An element $Q\in L$ is a {\em cancellation element} if $Qa=Qb$ with $a,b\in L$ implies $a=b$. In this case, 
$Qc \leq Qd$ with $c,d\in L$ implies $c \leq d$.
Say that a principal element $c\in L$ is {\em regular} if $(0:c)=0$. Clearly such an element is a cancellation element.

$L$ is a {\em $C$-lattice} if $1$ is compact, the set of compact elements  is closed under  multiplication  and  every element is a join of compact elements.
  In a $C$-lattice every principal element is compact.
 The $C$-lattices have a   localization theory similar to commutative rings.   Let $L$ be a $C$-lattice and $S$ a multiplicatively closed subset of $L$ consisting of compact elements.
The {\em localization} of $L$ at $S$ is
$$L_S = \{ a_S\ |\ a \in L \}
\mbox{ where }
a_S = \bigvee \{(a:s) \ |\  s \in S \}.$$  
Then $L_S$  is   a lattice under the following operations 
$$A\mapsto  (\bigvee A)_S, \ 
A\mapsto  \bigwedge A, \  
(a,b)\mapsto (ab)_S
\mbox{ for all }A\subseteq L_S,\ a,b\in L_S. $$
If $a,b\in L$ with $b$ compact, then $(a:b)_S=(a_S:b_S)$.
The localization $L_p$ of $L$ at a prime element $p$ is 
$L_T$ where $T= \{ c \mbox{ compact } \ |\ c\not\leq  p\}$.
For $a,b\in L$, we have $a=b$ iff $a_m=b_m$ for each maximal element $m$.
A compact element $x\in L$ is principal iff $x_m$ is principal in $L_m$ for each maximal element $m\in L$.

A subset $G$ of $L$ {\em generates} $L$ if every element of $L$ is a join of elements of $G$. 
A  modular  C-lattice generated by a set of principal elements is called an {\em r-lattice} (cf. \cite{A}). {\em This is  our working concept.} 
A typical example of an r-lattice is the ideal lattice of a commutative ring.

The key   step \cite[Lemma]{AR} in the proof of \cite[Theorem]{AR} 
is a modification of Kaplansky's argument  \cite[Theorem 287]{K}; it uses the  ideal equalities
$$(x^2,y^2)= (x^2,x^2+y^2)= (y^2,x^2+y^2)$$ 
where $x,y$ are elements of a commutative ring.
This fact motivates our ad-hoc definition below.

\begin{definition}\label{10}
Say that an r-lattice $L$ {\em has property  delta} if $L$ is generated by a set $\Delta$ of principal elements having the following property: 

if $x,y\in \Delta$, there exists some  $\delta\in \Delta$ (depending on $x,y$) such that
$$x^2\vee y^2=x^2\vee \delta = y^2 \vee \delta.$$
\end{definition}
 We state our main result which extends \cite[Theorem]{AR} to multiplicative lattices.

\begin{theorem} \label{7} 
Let $L$ be an r-lattice having property delta.
An element  $Q$ of $L$  is   a  cancellation element iff $Q_m$ is a principal regular element of $L_m$ for each maximal element $m$.
\end{theorem}

The proof is preceded by a string of five lemmas. We basically adapt the arguments in \cite{AR} to lattices.

\begin{lemma} \label{1} 
Let $L$ be an r-lattice and  $a,b,x\in L$ such that $x$ is principal   and $a\leq x\vee b$. Then $a\vee b = xa'\vee b$ for some $a'\in L$. 
\end{lemma} 
\begin{proof} 
As $L$ is  modular and $x$ is meet principal we have
$$ a\vee b =  (a\vee b) \wedge (x\vee b) = ((a\vee b) \wedge x)\vee b = xa'\vee b  \  \mbox{ for some } a'\in L.$$
\end{proof}

\begin{lemma} \label{2} 
Let $L$ be an r-lattice and  $Q,a,m\in L$ such that $Q$ is a  cancellation element, $m$ is maximal element and $Q=a\vee mQ$.
Then $Q_m = a_m$.
\end{lemma} 
\begin{proof} 
Let $b\leq Q$ be a principal element. 
As $Q$ is a cancellation element and $b$ is principal, we get successively  
$$Qb=ab\vee mQb\leq Q(a\vee mb),\  b\leq a\vee mb,\   1=(a:b) \vee m,\   (a:b)_m=1, \ b_m\leq a_m, $$ thus $Q_m = a_m$.
\end{proof}

\begin{lemma} \label{6} 
Let $L$ be an r-lattice,  $Q\in L$   a  cancellation element and  $m\geq Q$ a maximal element. Then $Q_m$ is regular in $L_m$.
\end{lemma} 
\begin{proof} 
Let $c\in L$ be a principal element with $cQ_m=0$. We can check locally that $cQ=mcQ$, so $c=mc$ (by cancelling $Q$), hence $1=m\vee (0:c)$, so $(0_m:c_m)=1$, thus $c_m=0$.
\end{proof}

\begin{lemma} \label{3} 
Let $L$ be an r-lattice,  $Q\in L$  a  cancellation element
and  $m\geq Q$ a maximal element. If $G$ is a set of principal elements generating $L$, there exists a  subset $B$ of $G$ such that $Q= \bigvee B \vee mQ$ and $B$ is minimal with this property, i.e.
$Q \neq \bigvee (B-\{b\}) \vee mQ$ for each $b\in B$.
\end{lemma} 
\begin{proof} 

The set $\Gamma$  of all subsets $C$ of $G$ such that 
$$ \bigvee C\leq Q \ \mbox{ and } \  
c\not\leq  \bigvee (C-\{c\}) \vee mQ \mbox{ for each }\  c\in C$$ 
is nonempty because $Q\neq mQ$. 

We claim that $\Gamma$ is inductively ordered under inclusion. Indeed let $\Lambda$ be a chain in $\Gamma$, $C=\bigcup \Lambda$ and $y\in C$.
Suppose that $y\leq \bigvee (C-\{y\}) \vee mQ$. As $y$ is compact and $\Lambda$ is a chain, we get 
$y\leq \bigvee (D-\{y\}) \vee mQ$ for some $D\in\Lambda$ 
containing $y$, a contradiction. Thus $\Gamma$ is inductively ordered.
By Zorn's Lemma, 
$\Gamma$ has a maximal element $B$. We show that $Q= \bigvee B \vee mQ$.
Deny. Pick a principal element $c\in G$, $c\leq Q$ such that $c\not\leq \bigvee B \vee mQ$.
As $B$ is maximal in $\Gamma$, we get 
$$b_0\leq  \bigvee (B-\{b_0\})\vee c \vee mQ \ \mbox{ for some }\  b_0\in B.$$
As $b_0$ is compact, we get $b_0 \leq  c\vee d_1\vee \cdots 
\vee d_n \vee mQ$ for some elements $d_1,...,d_n$ of $B-\{b_0\}$. Set 
$d=d_1\vee \cdots \vee d_n \vee mQ$. By Lemma \ref{1}, $b_0\vee d = xc\vee d$ for some $x\in L$. Since  $b_0\not\leq  \bigvee (B-\{b_0\}) \vee mQ$, we get $x\not\leq m$, so $x\vee m =1$. Then $c=xc\vee mc 
\leq b_0\vee d\vee mQ\leq \bigvee B  \vee mQ$, a contradiction.
\end{proof}

\begin{lemma} \label{4} 
Let $L$ be an r-lattice having property delta  with $\Delta\subseteq L$   as in Definition  \ref{10}. 
Let $Q\in L$ be  a  cancellation element such that 
$Q=x\vee y \vee a$ where $x,y\in \Delta$  and $a\in L$.
Suppose that $m(x\vee y) \leq a$ for some  maximal element $m\in L$.
Then $Q = x\vee a$ or $Q = y\vee a$.
\end{lemma} 
\begin{proof} 
By property delta, there exists $\delta\in \Delta$ such that $x^2\vee \delta = y^2\vee \delta = x^2\vee y^2.$
Set $b=xy \vee xa \vee ya \vee a^2$ and $j=\delta\vee b$.
Note that $j\leq Q^2$, so $jQ\leq Q^3$.
We have 
$$ Q^3 = \bigvee \{x^3, y^3, a^3, x^2y, xy^2, x^2a, 
xa^2, y^2a, ya^2, xya\}
$$
so $Q^3\leq jQ$ 
because $x^3\leq (\delta\vee y^2)x =\delta x\vee xy^2\leq jQ$ and similarly 
$y^3\leq jQ$. Then $Q^3=jQ$, so cancel $Q$ to get $Q^2=j$.
From $x^2\leq \delta \vee b$ we get $x^2\vee b = c\delta \vee b$ for some $c\in L$ (cf. Lemma \ref{1}). 
Note that $ m\delta  \leq mx^2 \vee my^2 
\leq xa \vee ya\leq b$.

If $c\leq m$, we have
$$x^2 \leq c\delta \vee b \leq m\delta \vee b  =  b \leq (y\vee a)Q$$
hence $Q^2=(y\vee a)Q$, so cancel $Q$ to get $Q=y\vee a$.

If $c\not \leq m$ hence $m\vee c =1$,
we have
$$ y^2 \leq   x^2\vee \delta \leq x^2 \vee m\delta  \vee c\delta \leq x^2\vee b\leq (x\vee a)Q
$$ 
hence  $Q^2=(x\vee a)Q$, so cancel $Q$ to get $Q=x\vee a$. 
\end{proof}

{\em Proof of Theorem \ref{7}.}
$(\Leftarrow)$
This direction works without property delta.
Let $b,c\in L$ such that $Qb=Qc$. It follows that  $b=c$ locally  (hence globally) because any principal regular element is a cancellation element. 

$(\Rightarrow)$ 
Let $\Delta\subseteq L$ be as in Definition  \ref{10} and $m$ a maximal element.
We may assume that $m \geq Q$. By Lemma \ref{3}, there exists a minimal  set $B\subseteq \Delta$   such that $Q= \bigvee B \vee mQ$. 
If $B$ has at least two elements $x,y$, apply Lemma \ref{4} for 
$a=\bigvee (B-\{x,y\})\vee mQ$ to get a contradiction.
So $B$  has only one element.
Then Lemmas \ref{2} and \ref{6} show that $Q_m$ is principal and regular. \ \ \ \ \ \ \ \ \ \ \ \ \ \ \ \ \ 
$\square$
\\[2mm]

We give a few consequences of Theorem \ref{7}.

\begin{corollary}
Let $L$ be an r-lattice having property delta and   $Q\in L$
   a   cancellation element. 

$(i)$ If $Q$ is compact, then  $Q$ is a principal regular element. 

$(ii)$ If $a,b\in L$, then $Q(a\wedge b) = Qa \wedge Qb.$

\end{corollary}
\begin{proof}
$(i)$ By Theorem \ref{7}, $Q$ is locally principal hence principal as $Q$ is compact.
$(ii)$ Checking locally, we may assume that $Q$ is a principal regular element. In this case the assertion is known.
\end{proof}
Note that ideal $A$ in \cite[Exercise 10, page 456]{G} is a finitely generated cancellation ideal which   consists of zero-divisors.
The next corollary extends \cite[Corollary 1]{AR}.

\begin{corollary}\label{15}
Let $L$ be an r-lattice having property delta and 
$S$  a multiplicatively closed set of compact elements in $L$.
Then $Q$ is a  cancellation element of $L_S$ 
iff $Q_p$ is a principal regular element of $L_p$ for each maximal element $p$ of the set 
$$T=\{ q \mbox{ prime element of } L\ |\  s\not\leq q 
\mbox{ for each } s\in S \}.$$
 In particular, if $Q$ is a  cancellation element of $L$, then  $Q_S$ is a  cancellation element of $L_S$
\end{corollary}
\begin{proof} 
It is easy to check  that $L_S$ is still an r-lattice having property delta.
Also, it is well-known that $T$ is the spectrum of $L_S$, so its
 maximal elements are precisely the maximal elements of $L_S$ (see \cite[Theorem 2.2]{A} and \cite[page 203]{JJ}).
If $p\in T$ and  $m\geq p$ is a maximal element of $L$, then $(L_S)_p=L_p$ and  $(Q_m)_p=Q_p=(Q_S)_p$ is a principal regular element in $L_p$ if so is $Q_m$ in $L_m$. Apply Theorem \ref{7} to complete.
\end{proof}

Let $D$ be an integral domain with quotient field $K$. We briefly recall the definition of the so called {\em $w$-operation} on $D$ 
(for   details see \cite{AC}, \cite[Sections 32 and 34]{G} and  \cite{WM}). Let $W$ be the set of nonzero finitely generated  ideals $H$ of $D$ such that 
$H^{-1}=D$ where $H^{-1}:=\{x\in K\ |\ xH\subseteq D\}$. For an ideal $Q$ of $D$ define its {\em $w$-closure} by $$Q_w=\cup \{(Q:H)\ |\ H\in W\}$$ and call $Q$ a {\em $w$-ideal} if $Q=Q_w$. Say that  a $w$-ideal $Q$ a 
{\em $w$-cancellation} ideal if 
$$(QJ)_w =(QN)_w \mbox{ with } J,N \mbox{ ideals implies } J_w = N_w.$$ 
The set $\mathcal{L}_w(D)$ of all $w$-ideals of $D$ is 
an r-lattice  under  the operations: 
$$\bigvee A_\alpha = (\sum A_\alpha)_w,\  \bigwedge A_\alpha = \cap A_\alpha,\  A\cdot B=(AB)_w$$
 cf. \cite[Theorems 3.1 and 3.5]{AC}. The maximal elements of $\mathcal{L}_w(D)$ are called maximal $w$-ideals.
Moreover $\mathcal{L}_w(D)$ has property delta 
since  $(x^2,y^2)_w= (x^2,x^2+y^2)_w= (y^2,x^2+y^2)_w$ for all $x,y\in D$.
By Theorem \ref{7} we get:

\begin{corollary}\label{16}
Let $D$ be an integral domain. A nonzero w-ideal $Q$ of $D$ is a w-cancellation ideal iff $QD_M$ is  principal  for each maximal $w$-ideal $M$ of $D$.
\end{corollary}
A  shorcut is to note that $W$ is a multiplicative set of compacts  in the ideal lattice $\mathcal{L}(D)$ of  $D$ and 
$\mathcal{L}_w(D)$ is just the localization of $\mathcal{L}(D)$ at $W$. 
So Corollary \ref{16} follows at once from Corollary \ref{15}.

\begin{remark}\label{11}
The discussion  above can be plainly extended from $w$ to star operation $*_w$ where $*$ is a star operation on $D$ (see for instance \cite{C}). Doing so we retrieve \cite[Theorem 2.3]{C}.
\end{remark}

We end this note  with the following example. The ideal lattice  $L$ of the semiring $\mathbb{N}$  is neither modular  
(cf. \cite[page 582]{AA}) nor with property delta (as the principal elements of $L$ are the principal ideals of $\mathbb{N}$ and
there is no $c\in \mathbb{N}$ such that $(4,9)=(4,c)=(c,9)$).
Nevertheless, any cancellation  ideal   of $\mathbb{N}$ is   principal. Indeed, a non-principal ideal $Q$ can be written as $(a,b,H)$ where $a < b$ are positive integers, $b$ is not a multiple of $a$ and $H$ is an ideal generated by numbers greater than $b$. Hence  $Q^3=Q(a^2,b^2,aH,bH,H^2)$,  $ab\in Q^2$ and 
$ab\notin (a^2,b^2,aH,bH,H^2)$.

\end{document}